\documentclass{amsart}
\usepackage{amssymb}
\usepackage{latexsym}
\usepackage{euscript}
\usepackage[english]{babel}
\def\cH{{\mathcal H}}

\def\cD{{\mathcal D}}
\def\cG{{\mathcal G}}

\def\cQ{{\mathcal Q}}
\def\cF{{\mathcal F}}
\def\dD{{\mathbb D}}

\def\dT{{\mathbb T}}
\def\sH{{\mathfrak H}}
 \def\sF{{\mathfrak F}}
\def\sG{{\mathfrak G}}

\def\sL{{\mathfrak L}}

\newtheorem{theorem}{Theorem}[section]
\newtheorem{proposition}[theorem]{Proposition}
\newtheorem{corollary}[theorem]{Corollary}
\newtheorem{lemma}[theorem]{Lemma}

\newtheorem{definition}[theorem]{Definition}

\numberwithin{equation}{section}

\renewenvironment {proof} {\begin{trivlist} \item[\hspace{\labelsep}%
\sc Proof.]}{$\Box$ \end{trivlist}}

\newcommand {\sk}[3]{\left#1#2\right#3}  

\newcommand {\wh}{\widehat}
\renewcommand {\l}{\lambda}
\renewcommand {\k}{\kappa}
\newcommand {\s}{\sigma}

\newcommand {\e}{\varepsilon}

\newcommand {\g}{\gamma}

\newcommand {\ov}{\overline}
\newcommand {\m}{\mu}
\newcommand {\G}{\Gamma}
\newcommand {\D}{\Delta}
\renewcommand {\r}{\rho}

\newcommand {\Pp}{\mathcal{P}_{\mathfrak{L}_2}}

\newcommand {\w}{\widetilde}
\newcommand {\p}{\perp}

\def\gr{{\rm gr}}
\def\sq{{\rm sq}}
\def\rank{{\rm rank}}
\def\span{{\rm span}}                  
\def\ran{{\rm ran\,}}

\def\ind{{\rm ind\,}}
\def\dom{{\rm dom\,}}
\def\ker{{\rm ker\,}}

\def\mul{{\rm mul\,}}
\newcommand {\zx}{{[*]}}

\begin{document}
\title[Abstract interpolation problem in generalized Schur classes.]
{Abstract interpolation problem in generalized Schur classes.}
\author{D.~Baidiuk}
\address{ Department of Mathematics \\ Donetsk National University \\ 24, Universitetskaya Str., 83001 Donetsk}

\address{Department of Mathematics and Statistics \\
University of Vaasa \\
P.O. Box 700, 65101 Vaasa \\
Finland} \email{dbaidiuk@uwasa.fi}



\begin{abstract}
An  indefinite variant of the abstract interpolation problem is
considered. Associated to this problem is a model Pontryagin space
isometric operator $V$. All the solutions of the problem are shown
to be in a one-to-one correspondence with a subset of the set of all
unitary extenions $U$ of $V$. These unitary extension $U$ of $V$ are
realized  as unitary colligations with the indefinite de
Branges-Rovnyak space $\cD(s)$ as a state space.
\end{abstract}

\maketitle

\section{Introduction.}

The abstract  interpolation problem in the Schur class $S$ have been
posed and considered by V.~Katsnelson, A.~Khejfets and P.~Yuditskij
\cite{KKhYu87}. It contains the most classical interpolation
problems such as the moment problem, the bitangential
Schur-Nevanlinna-Pick problem and others (see \cite{KhYu94},
\cite{Kh90}, \cite{Kh96} and \cite{Yu83}). The method of abstract
interpolation problem contains and develops ideas of the
V.~P.~Potapov's approach to interpolation problems \cite{KP74}, the
theory of  unitary colligation \cite{Br78}, and the theory of
reproducing kernel Hilbert spaces \cite{dBrR66}. In particular, the
results of D.~Z.~Arov and L.~Z.~Grossman on scattering matrices of
unitary operators \cite{AG92} were used in order to describe the set
of solutions of this problem. These results are closely related to
the M.G.~Kre\u{\i}n's theory of ${\sL}$-resolvent matrices for
symmetric operators \cite{Kr44} and \cite{KrS66} (see also
\cite{Sh70} and \cite{DM95}).

The present paper deals with the indefinite abstract interpolation
problem $AIP(\kappa)$, $\kappa\in{\Bbb Z}_+$ in generalized Schur
classes (see definition below). It is shown that this problem can be
reduced to the extension problem for a model Pontryagin space
isometric operator $V$, associated with the problem $AIP(\kappa)$.
As distinct from the Hilbert space case the correspondence between
minimal unitary extensions $U$ of $V$ and their scattering matrices
does not give anymore a parametrization of the solution set of the
problem $AIP(\kappa)$. The desired description is given in Section 5
by selecting of a subclass of the so-called ${\sL}$-regular unitary
extensions of $V$. The unitary extension $U$ of $V$ is realized in
the paper as a unitary colligation with the indefinite de
Branges--Rovnyak space $\cD(s)$ as a state space. The corresponding
construction is very close to that given in \cite{ADRS91} and
\cite{D2001}.
 The statement of the abstract interpolation problem in the present paper is different from the statement of
 this problem in \cite{D2001}. The problem data in this paper contain two
 different operators $M$ and $N$ whereas in the paper \cite{D2001} one
 of them equals an identity operator. The description of solutions
 of the Problem $AIP(\k)$ in the present formulation can be used for getting a
 description of the bitangential interpolation problem.

\section{Preliminaries.}
\begin{subsection}{Linear relations.}
Let $\cH_1$, $\cH_2$ be Hilbert spaces. A linear manifold
$T\subset\cH_1\oplus\cH_2$ is called a linear relation (shortly
l.r.) in $\cH_1\oplus\cH_2$ (from $\cH_1$ to $\cH_2$). We denote by
$\w{ \mathcal{C}}(\cH_1,\cH_2)$ $(\w {\mathcal{C}}(\cH))$ the set of
all closed linear relations in $\cH_1\oplus\cH_2$ (in
$\cH\oplus\cH$). For a linear relation $T\subset\cH_1\oplus\cH_2$ we
denote by $\dom T$, $\ran T$, $\ker T $ and $\mul T$ the domain, the
range, the kernel and the multivalued part of $T$ respectively.

If $T$ is a linear relation in $\cH_1\oplus\cH_2$, then the inverse
$T^{-1}$ and adjoint $T^*$ relations are defined as
$$
{T^{-1}=\sk\{{\begin{bmatrix} f' \\ f
\end{bmatrix}:\begin{bmatrix} f \\ f' \end{bmatrix}\in T}\}},\
T^{-1}\subset\cH_2\oplus\cH_1
$$
$$
T^*={\begin{bmatrix} g \\ g'
\end{bmatrix}\in\cH_1\oplus\cH_2:(f',g)=(f,g'), \begin{bmatrix} f \\ f'
\end{bmatrix}\in T},\ T^*\in\w {\mathcal{C}}(\cH_2,\cH_1).
$$
A closed linear operator $T$ from $\cH_1$ to $\cH_2$ is identified
with its graph $\gr T\in\w {\mathcal{C}}(\cH_1,\cH_2)$.

In the case $T\in\w {\mathcal{C}}(\cH_1,\cH_2)$ we write:

$0\in\r(T)$ if $\ker T={0}$ and $\ran T=\cH_2$;

$0\in\wh\r(T)$ if $\ker T={0}$ and $\ov\ran T=\ran T\neq\cH_2$;

$0\in\s_c(T)$ if $\ker T={0}$ and $\ov\ran T=\cH_2\neq\ran T$;

$0\in\s_p(T)$ if $\ker T\neq{0}$;

$0\in\s_r(T)$ if $\ker T\neq{0}$ and $\ov\ran T\neq\cH_2$.

For a l.r. $T\in\w {\mathcal{C}}(\cH)$ we denote by
$\r(T)=\{\l\in\mathbb{C}:0\in\r(T-\l)\}$ and
$\wh\r(T)=\{\l\in\mathbb{C}:0\in\wh\r(T-\l)\}$ the resolvent set and
the set of regular type points of $T$ respectively. Next,
$\s(T)=\mathbb{C}\setminus\r(T)$ stands for the spectrum of $T$.
\end{subsection}

\begin{subsection}{Linear relations in Pontryagin spaces.}

In this subsection we review some facts and notation from
\cite{AI86,Bog74}. Let $\mathcal{H}$ be a Hilbert space and
$j_{\mathcal{H}}$ be a signature operator in this space, (i.e.,
$j_{\mathcal{H}}=j_{\mathcal{H}}^*=j_{\mathcal{H}}^{-1}$). The space
$\mathcal{H}$ can be considered as a Kre\u{\i}n space
$(\mathcal{H},j_{\mathcal{H}})$ (see \cite{AI86}) with the inner
product $ [\varphi, \psi]_{\mathcal{H}}=
 (j_{\mathcal{H}} \varphi, \psi)_{\mathcal{H}}.
$ The signature operator $j_\mathcal{H}$ can be represented as
$j_\mathcal{H}=P_+-P_-$, where $P_+$ and $P_-$ are orthogonal
projections in $\mathcal{H}$. In the case when
 $\mbox{dim}P_-\mathcal{H}=\kappa<\infty$, the Kre\u{\i}n space
$(\mathcal{H},j_\mathcal{H})$ is called a Pontryagin space with the
negative index $\kappa$ and is denoted by
$\mbox{ind}_-\mathcal{H}=\kappa$.

 Let us consider two Pontryagin spaces $(\mathcal{H}_1,j_{\mathcal{H}_1})$, $(\mathcal{H}_2,j_{\mathcal{H}_2})$ and a linear relation
 $T$ from $\mathcal{H}_1$ to $\mathcal{H}_2$. Then an adjoint l.r.
 $T^{[*]}$ consists of pairs $\begin{bmatrix} g_2 \\g_1
\end{bmatrix}\in\cH_2\times\cH_1$ such that
$$
[f_2,g_2]_{\cH_2}=[f_1,g_1]_{\cH_1}\text{ for all }\begin{bmatrix}
f_1
\\f_2
\end{bmatrix}\in T.
$$
If $T^*:\cH_2 \to\cH_1$ is an adjoint linear relation of $T$ in the
Hilbert spaces $\cH_1$ and $\cH_2$ then
$T^{[*]}=j_{\cH_1}T^*j_{\cH_2}$.

The l.r. $T^{[*]}$ satisfies the following equations
\begin{equation}
\label{adjequ}
 (\dom T)^{[\perp]}=\mul T^{[*]}, \quad
 (\ran T)^{[\perp]}=\ker T^{[*]},
\end{equation}
 where the sign $[\perp]$  denotes orthogonality in Pontryagin spaces.
\begin{definition}\label{D:2.1}
 A l.r. $T$ from a Pontryagin space $(\cH_1,j_{\cH_1})$ to a
 Pontryagin space
  $(\cH_2,j_{\cH_2})$ is called an isometry if the equality
\begin{equation}
\label{iso}
 [ \varphi', \varphi']_{\cH_2} = [ \varphi, \varphi]_{\cH_1},
\end{equation}
holds for every $\begin{bmatrix} \varphi \\\varphi'
\end{bmatrix}\in
  T$ and it is called a contraction if the sign in the \eqref{iso} is substituted for $\le$. A l.r. $T$ is called
   a unitary l.r. from
  $(\cH_1,j_{\cH_1})$ to $(\cH_2,j_{\cH_2})$ if $T^{-1}
= T^{[*]}$. Clearly, a l.r. $T$ is an isometric l.r. if and only if
$T^{-1} \subset T^{[*]}$.
\end{definition}
We recall (\cite{AI86}) that the sets $\dD\setminus\wh\r(T)$ and
$\dD_e\setminus\wh\r(T)$ for an isometric operator $T$ in a
Pontryagin space $\ind_-\cH=\k$ consist of at most $\k$ points
belonging to $\s_p(T)$.

 The definition of unitary l.r. at first was introduced in
 \cite{Sh76}. In particular, in  \cite{Sh76} the following Proposition was proved
\begin{proposition}\label{Plo}
If $T$ is a unitary relation  from a Pontryagin space
$(\cH_1,j_{\cH_1})$ to a
 Pontryagin space
  $(\cH_2,j_{\cH_2})$ then
\begin{enumerate}
\item
$\dom T$ is closed if and only if \\$\ran T$ is closed;
\item
the equalities $\ker T=\dom T^{[\p]}$, $\mul T=\ran T^{[\p]} $ hold.
\end{enumerate}
\end{proposition}
From Proposition \ref{Plo}, we get
\begin{corollary}
If $T$ is a unitary l.r. in a Pontryagin space then the condition
$\mul T\neq\{0\}$ is equivalent to the condition $\ker T\neq\{0\}$.
Moreover,  the equality $\dim\mul T=\dim\ker T$ holds.
\end{corollary}

\end{subsection}

\begin{subsection}{The generalized Schur class.}

Recall that a Hermitian kernel
$K_\omega(\l):\Omega\times\Omega\to\mathbb{C}^{m\times m}$ is said
to have $\k$ negative squares if for every positive integer $n$ and
every choice of $\l_j \in \Omega$ and $u_j \in \mathbb{C}^m$
$(j=1,...,n)$ the matrix
\begin{equation*}
\left( [K_{\l_j}(\l_k)u_j,u_k]_{\cH} \right)_{j,k=1}^n
\end{equation*}
has at most $\kappa$ negative eigenvalues and for some choice of
 $\l_1,...,\l_n \in \Omega$ and $u_1,...,u_n \in \mathbb{C}^m$ exactly $\k$ negative eigenvalues.
 In this case we write
 $$
\sq_-K=\k.
 $$

Let $\k$ be a nonnegative integer, $\sL_2=\mathbb{C}^q$,
$\sL_1=\mathbb{C}^p\ (p,q\in\mathbb{Z}_+)$ and let
$\mathcal{B}(\sL_2,\sL_1)$ be the set of $p\times q$-matrices. A
$\mathcal{B}(\sL_2,\sL_1)$ valued function holomorphic in a
neighborhood of $0$ is said to belong to the generalized Schur class
$S_\k(\sL_2,\sL_1)$ in the unit disc if the kernel
\begin{equation}\label{kerLambda}
{\mathsf
\Lambda}_\omega^s(\l)=\frac{I_{p}-s(\l)s(\omega)^*}{1-\l\ov\omega}\quad(\l,w\in\Omega_s\subset\dD)
\end{equation}
has $\kappa$ negative square on $\Omega_s$ (see ~\cite{KrL72}). The
class $S(\sL_2,\sL_1):=S_0(\sL_2,\sL_1)$ consists of usual Schur
function.
 An example of a generalized Schur function is provided by the Blaschke-Potapov product
\begin{equation}\label{BPprod}
b(\lambda)=\prod b_j(\lambda),\quad
b_j(\lambda)=I-P_j+\frac{\lambda-\alpha_j}{1-\overline{\alpha}_j\lambda}P_j,
\end{equation}
where $\alpha_j \in \dD$, $P_j$ orthoprojections in $\mathbb{C}^p$
$(j=1,...,k)$. The factor $b_j(\cdot)$ is called simple if $P_j$ has
rank one. Although $b(\cdot)$ can be written as a product of simple
factors in many ways; the number of this factors is the same for
every representation \eqref{BPprod}. It is called the degree of the
Blaschke-Potapov product $b(z)$ \cite{P55}.

A theorem of Kre\u{\i}n-Langer \cite{KrL72} guarantees that every
generalized Schur function $s(\cdot)\in S_{\kappa}^{p\times
q}(\sL_2,\sL_1)$ admits a factorization of the form
\begin{equation}\label{KLleft}
s(\l)=b_l(\l)^{-1}s_l(\l) \quad (\l\in\Omega_s),
\end{equation}
where $b_l(\cdot)$ is a Blaschke-Potapov product of degree $\kappa$,
 $s_l(\cdot)$ is in the Schur class $S(\sL_2,\sL_1)$ and
\begin{equation}\label{KLcanon}
\ker s_l(\l)^*\cap \ker b_l(\l)^*=\{0\}\quad (\l\in\Omega_s).
\end{equation}
The representation ~\eqref{KLleft} is called a left
Kre\u{\i}n-Langer factorization. The constraint ~\eqref{KLcanon} can
be expressed in the equivalent form
\begin{equation}\label{KRcanon2}
\rank [b_l(\l), s_l(\l)]=p \quad (\l\in\Omega_s).
\end{equation}

If $\alpha_j \in \dD$ $(j=1,...,n)$ are all zeros of $b_l(\cdot)$ in
$\dD$, then the condition ~\eqref{KLcanon} ensures that
$\Omega_s=\dD\setminus \{\alpha_1,...,\alpha_n\}$. The left
Kre\u{\i}n-Langer \eqref{KLleft} is essentially unique in a sense
that $b_l(\cdot)$ is defined uniquely up to a left unitary factor
$V\in \mathbb{C}^{p\times p}$.

Similarly, every generalized Schur runction  $s(\cdot)\in
S_{\kappa}^{p\times q}(\sL_2,\sL_1)$ a right Kre\u{\i}n-Langer
factorization
\begin{equation}\label{KLright}
s(\l)=s_r(\l)b_r(\l)^{-1}, \quad(\l \in \Omega_s),
\end{equation}
where $b_r(\cdot)$ is a Blaschke-Potapov product of degree $\kappa$,
$s_r(\cdot)$is in the Schur class $S(\sL_2,\sL_1)$ and
\begin{equation}\label{KLcanonR}
\ker s_r(\l)\cap \ker b_r(\l)=\{0\}\quad (\l\in\Omega_s).
\end{equation}
This condition can be rewritten in the equivalent form
\begin{equation}\label{KLcanon2R}
\rank [b_r(\l), s_r(\l)]=q \quad (\l\in\Omega_s).
\end{equation}
Under assumption \eqref{KLright} the matrix valued function
$b_r(\cdot)$ is uniquely defined up to a right unitary factor $V'
\in \mathbb{C}^{q\times q}$.

Let $\Pi_+$ and $\Pi_-$ denote the orthogonal projections from
$L_2^k$ onto $H^k_2$ and $(H^k_2)^\p$ respectively, where $k$ is a
positive integer that will be understood from the context. Let us
introduce the Hilbert spaces
\begin{equation}\label{Hb}
\cH(b_r):=H_2^q\ominus b_rH_2^q,\quad\cH_*(b_l):=(H_2^p)^\p\ominus
b_l^*(H_2^p)^\p
\end{equation}
and the operators
\begin{equation}\label{Xrl}
X_r:h\in\cH(b_r)\to\Pi_-sh,\quad X_l:h\in\cH_*(b_l)\to\Pi_+s^*h
\end{equation}
based on $s(\cdot)$.

The next operators will play an important role.
\begin{definition}
Let
\begin{equation}\label{Gamma l}
\G_l:f\in L_2^q\to X_l^{-1}P_{\cH(b_r)}f\in\cH_*(b_l);
\end{equation}
\begin{equation}\label{Gamma r}
\G_r:g\in L_2^p\to X_r^{-1}P_{\cH_*(b_l)}g\in\cH(b_r),
\end{equation}
where $X_l$ and $X_r$ are defined by ~\eqref{Xrl}.
\end{definition}
This operators $\G_r$ and $\G_l$ are using for introducing a metric
in some space.
\end{subsection}

\begin{subsection}{Unitary colligations in Pontryagin spaces.}
The theory of unitary colligations in Hilbert spaces was introduced
in the paper \cite{L46} and had further development in the papers
\cite{BrGK70,Br78}. The theory of unitary colligations in Pontryagin
spaces was built in \cite{Kuzh96} and \cite{ADRS91}, in the latter
the functional models of these colligations were studied.

In the present paper we use the notation $\wh\cD(s)$ which is
different from that used in the papers \cite{ADRS91} and
\cite{D2001}.

 Let us recall  some basic
notions from the theory of unitary colligations (see \cite{Br78},
\cite{DLS}). Let $\cH$ be a Pontryagin space with the negative index
$\kappa$ (see [4,5]),
 let $\sL_1$, $\sL_2$ be Hilbert spaces,
and let $U= \begin{bmatrix}
T & F\\
G & H
\end{bmatrix}
$ be a unitary operator from $\cH\oplus\sL_2 $ into
$\cH\oplus\sL_1$. Then the quadruple $\Delta=(\cH,\sL_2,\sL_1, U)$,
where  $\cH$ denotes the so--called state space and $\sL_2$, $\sL_1$
stand for the incoming and outgoing spaces, respectively, is said to
be a {\it unitary colligation}.

The colligation $\Delta$ is said to be {\it simple}, if there is no
reducing subspace $\cH_1\subset\cH$. The colligation $\Delta$ is
simple (see \cite{DLS}) if and only if
\begin{equation}\label{span1}
(\cH_\Delta:=)\ov{\hbox{span}}\left\{T^nFh_2,{T^{[*]}}^nG^{[*]}h_1:\,\,
h_1\in\sL_1,h_2\in\sL_2,\,\,n\in{\Bbb Z}_+\right\}=\cH.
\end{equation}
The operator valued function
\begin{equation}\label{ChF}
 s(z)=H+\l G(I-\l
T)^{-1}F:\sL_2\to\sL_1\quad (1/\l\in\rho(T))
\end{equation}
 is said to be the
characteristic function of the colligation (or the scattering matrix
of the unitary operator $U$ with respect to the channel subspaces
$\sL_2$, $\sL_1$ see \cite{AG92}). If the colligation $\Delta$ is
simple then $s\in S_{\kappa}(\sL_2,\sL_1)$. One can rewrite the
formula \eqref{ChF} in the form
\begin{equation}\label{harfun}
s(\l )= P_{\sL_1}U(I-\l P_\cH U)^{-1}P_{\sL_2}
     =P_{\sL_1}(I-\l UP_\cH )^{-1}UP_{\sL_2},
\end{equation}
since
$$
(I-\l P_\cH U)^{-1} = \begin{bmatrix}
(I-\l T)^{-1} & \l(I-\l T)^{-1}F\\
0 & I
\end{bmatrix} ,
$$
\begin{equation}\label{blok-matr}
U(I-\l P_\cH U)^{-1}= \begin{bmatrix}
T(I-\l T)^{-1} & F+\l T(I-\l T)^{-1}F\\
G(I-\l T)^{-1} & H+\l G(I-\l T)^{-1}F
\end{bmatrix}.
\end{equation}
Here $P_{\cH}$, $P_{\sL_i}$ are orthogonal projections from
$\cH\oplus\sL_i$ onto $\cH$ and $\sL_i$ $(i=1,2)$, respectively.
\end{subsection}

\begin{subsection}{The de Branges-Rovnyak space $\cD(s)$.}
The symbol $A^{[-1]}$ stands for the Moore-Penrose pseudoinverse of
the matrix $A$ (см. \cite{G88}),
\begin{equation}\label{D(s)1}
\Delta_s(\m):=\begin{bmatrix}
I_p&-s(\m) \\
-s(\m)^*&I_q
\end{bmatrix}
\end{equation}
a.e. on $\mathbb{T}$ for $s(\cdot)\in S_\k^{p\times q}$.
\begin{definition}\label{dBR}
Let a matrix valued function $s(\cdot)\in S_\k^{p\times q}$ admit
left and right Kre\u{\i}n-Langer factorizations \eqref{KLleft} and
\eqref{KLright}. Define $\cD(s)$ as the set of vector valued
functions $f(t)=\begin{bmatrix}
f_+(t) \\
f_-(t)
\end{bmatrix}$ such that the following conditions hold:
\begin{enumerate}
\item
$b_lf_+\in H_2^p;$
\item
$b_r^*f_-\in (H_2^q)^\p;$
\item
$f(t)\in\ran\Delta_s(\m)$ a.e. on $\mathbb{T}$ and the following
integral
$$
\frac{1}{2\pi}\int\limits_\mathbb{T}f(\m)^*\Delta_s(\m)^{[-1]}f(\m)d\m
$$
converges.
\end{enumerate}
The inner product in $\cD(s)$ is defined by
\begin{equation}\label{D(s)2}
[
f,g]_{\cD(s)}=\frac{1}{2\pi}\int\limits_\mathbb{T}g(\m)^*\left(\Delta_s(\m)^{[-1]}+\begin{bmatrix}
0&\G_r^* \\
\G_r&0
\end{bmatrix}\right)f(\m)d\m,
\end{equation}
where the operator $\G_r$ was defined in \eqref{Gamma r}.
\end{definition}
As has been shown in \cite{DD09} the space $\cD(s)$ is a Pontryagin
space with the negative index $\k$. In the case when $\k=0$ the
space $\cD(s)$ was introduced in \cite{dBrR66} (see also
\cite{KKhYu87}).
\end{subsection}

\begin{subsection}{The generalized Potapov class and generalized $J$-inner functions.}
Let $\k, m\in \mathbb{N}$ and $J$ be a $m\times m$ signature matrix
( i.e., $J=J^*$ и $JJ^*=I_m$).
\begin{definition}
Recall that a meromorphic in $\dD$ $m\times m$-valued matrix
function $W(\cdot)$ belongs the generalized Potapov class
 $\mathcal{P}_\k(J)$ \cite{AD86}, if the kernel
\begin{equation}\label{KerK}
K_\omega^W(\l)=\frac{J-W(\l)JW(\omega)^*}{1-\l\ov\omega}
\end{equation}
has $\k$ negative squares in $\Omega_W$, where $\Omega_W$ is the
domain of holomorphy of $W$ in $\dD$.
\end{definition}
\begin{definition}
\cite{AD86,DD09} A meromorphic in $\dD$ $m\times m$-valued matrix
function $W(\cdot)$ is called $J$-inner matrix valued function (it
is denoted by $W\in\mathcal{U}_\k(J)$), if it belongs the
generalized Potapov class $\mathcal{P}_\k(J)$ and
\begin{equation}\label{Jin}
J-W(\m)JW(\m)^*=0
\end{equation}
a.e. $\m\in \mathbb{T}$. This class is denoted by
$\mathcal{U}_\k(J)$.
\end{definition}

\end{subsection}
\begin{subsection}{Reproducing kernel Pontryagin spaces.}
A Pontryagin space $(\cH, [\cdot , \cdot]_{\cH})$ of
$\mathbb{C}^m$-valued functions defined in a subset $\Omega$ of
$\mathbb{C}$ is called a reproducing kernel Pontryagin space if
there exists a Hermitian kernel $K_\m(\l): \Omega\times\Omega
\rightarrow \mathbb{C}^{m\times m}$ such that

\begin{enumerate}
    \item $K_\m(\l)u \in \cH$, $ \text{ for every } \m\in \Omega,$ $u \in \mathbb{C}^m$;
    \item $[f, K_\m u]_{\cH}=u^*f(\m),$ $\text{ for every } f(\cdot)\in \cH,$ $\m\in \Omega$,
$u\in\mathbb{C}^m$.
\end{enumerate}

It is known (see \cite{Sch64}) that for every Hermitian kernel
$K_\m(\l)\Omega\times\Omega\to \mathbb{C}{m\times m}$ with a finite
number of negative squares on $\Omega\times\Omega$ there is a unique
Pontryagin space $\cH$ with reproducing kernel $K_\m(\l)$, and that
$\ind_-\cH =\sq_-K =\kappa$. In the case $\k=0$ this fact is due to
Aronszajn \cite{A50} (see also \cite{D2001}).
\end{subsection}
\begin{subsection}{Space $\w\cD(s)$}
Let $s \in S_\kappa^{p\times q}(\dD)$ be the characteristic function
of a unitary colligation $\Delta=(\cH,\sL_2,\sL_1;T,F,G,H)$. Let us
consider the kernel $D_s(\lambda,\m)$ on $\Omega_s\times \Omega_s$
defined by the matrix
\begin{equation}\label{Ds}
D_s(\m,\l)=\begin{bmatrix}
            \frac{I_{\sL_1}-s(\l)s(\m)^*}{1-\l\ov\m}&-\m\frac{s(\l)-s(\m)}{\l-\m}\\
            -\ov\l\frac{s(\l)^*-s(\m)^*}{\ov\l-\ov\m}&\ov\l\m\frac{I_{\sL_2}-s(\l)^*s(\m)}{1-\ov\l\m}
           \end{bmatrix},\ (\l,\m\in\Omega(s)).
\end{equation}
The introduced kernel is similar to the kernels from \cite{D2001}
and \cite{DD09}. A Pontryagin space corresponding to this
reproducing kernel $D_s(\lambda,\m)$ is denoted by $\w\cD(s)$.

Let us define two operator functions $G_1(z):{\sL}_1\to{ \cH}$ и
$G_2(z): {\sL}_2\to{ \cH}$
\begin{equation}\label{G1G2}
\begin{split}
G_1(\l)^\zx  &=P_{{\sL}_1}U(I-\l P_{\cH} U)^{-1}|_{{\cH}}\quad
(1/\l\in\rho(T)),\\
G_2(\l)  &=-P_{{\cH}}U(I-\l P_{\cH} U)^{-1}|_{{\sL}_2}\quad
(1/\l\in\rho(T)).
\end{split}
\end{equation}
It follows from \eqref{blok-matr}: that
$$
G_1(\l)^\zx=G(I-\l T)^{-1},\quad G_2(\l)=-(I-\l T)^{-1}F\quad
(1/\l\in\rho(T))
$$
and the formula \eqref{span1} for the subspace ${\cH}_\Delta$ can be
rewritten as
$$
{\cH}_\Delta=\overline{\hbox{span}}\left\{G_1(\l)h_1,G_2(\l)h_2:\,\,h_1\in{\sL}_1,
        h_2\in{\sL}_2,\,\,1/\l\in\rho(T)\right\}.
$$
As is easily checked for every $\w f=\begin{bmatrix}
        \w f_1 \\
        \w f_2
       \end{bmatrix},\
\w g=\begin{bmatrix}
       \w g_1 \\
       \w g_2
      \end{bmatrix}
\in \sL_1\oplus\sL_2$ the following identity holds
\begin{equation}\label{fact kern}
(D_s(\m,\l)\w f,\w g)_{\sL_1\oplus\sL_2}=[G_1(\m)\w f_1+\m G_2(\m)\w
f_2,G_1(\l)\w g_1+\l G_2(\l)\w g_2]_\cH
\end{equation}
It follows from \eqref{fact kern} that the kernel $D_s(\m,\l)$ has
at most $\k$, and if the colligation $\Delta$ is simple exactly
$\k$, negative squares on $\Omega(s)$.

The next Theorem is the reformulation of the Theorem from the paper
\cite{DD09} for the reproducing kernel space $\w\cD(s)$.
\begin{theorem}\label{hatD} Let $s\in
S_\k^{p\times q}$, then the de Branges-Rovnyak space $\cD(s)$ is
unitarily equivalent to the reproducing kernel space $\w\cD(s)$ via
the mapping эквивалентность устанавливается отображением
\begin{equation}\label{DhatD}
\mathcal{T}: \w f=\begin{bmatrix}
       \w f_1 \\
       \w f_2
      \end{bmatrix}
      \in \w\cD(s)\to
      f=\begin{bmatrix}
         f_+ \\
         f_-
        \end{bmatrix}\in\cD(s),
\end{equation}
where $\w f_1$ is the meromorphic continuation of $f_+$ to
$\Omega_s$, and $\w f_2^*$ is the meromorphic continuation of
$f_-^*$ to $\Omega_s$ such that $\w f$ is a nontangential limit of
$f$ from the unit disk.
\end{theorem}
\end{subsection}

\section{Functional model of a unitary colligation $\Delta$.}
 In this section we will define the Fourier representation of an
unitary colligation and recall its functional model like in
\cite{D2001}.

 Recall (see \cite{AI86}) that a subspace
$\cH_1$ of the Pontryagin space $\cH$ is called regular if it is
orthocomplemented.
\begin{proposition}\label{PropFM1}
Let $\Delta=(\cH,\sL_2,\sL_1;U)$ be a unitary colligation such that
$\cH$ is a Pontryagin space and let $s(\cdot)$ be the corresponding
characteristic function. If $\cH_\Delta$ is a regular subspace of
$\cH$ then the space $\cD(s)$ can be identified with the space $\cD$
of vector functions
\begin{equation}\label{Fur'e}
(\cF h)(\l)=\begin{bmatrix}
              G_1(\l)^\zx h  \\
              \ov\l G_2(\l)^\zx h
            \end{bmatrix}      =\begin{bmatrix}
                                 G(I-\l T)^{-1}h \\
                                 -\l F^\zx(I-\ov \l T^\zx)^{-1}
                                \end{bmatrix}                 \ (h\in\cH_\D),
\end{equation}
equipped with the inner  product
\begin{equation}\label{skalarPr}
[\cF h,\cF g]_{\cD(s)}=[h,g]_\cH\ (h,g\in\cH_\D).
\end{equation}
\end{proposition}
\begin{proof}
As was mentioned above the kernel $D_s(\m,\l)$ has a finite number
of negative squares. In view of \eqref{fact kern} the equality
\eqref{Ds} can be rewritten in the form
\begin{equation}\label{DFF}
D_s(\m,\l)=\cF(\l)\cF(\m)^\zx\ (\l,\m\in\Omega_s).
\end{equation}
Hence the function $D_s(\m,\cdot)x$ belongs to $\cD$ for every
$x\in\sL_1\oplus\sL_2$. The mapping $\cF:\cH_\D\to\cD$ is
one-to-one, since $\cF h(\l)\equiv0$ implies $h\p\cH_\D$ due to
\eqref{skalarPr}, and, therefore, $h=0$. Moreover, it follows from
\eqref{skalarPr} and \eqref{DFF} that $\cD$ is isometrically
isomorphic to the space $\cH_\D$ and the following reproducing
property of the kernel $D_s(\m,\l)$ holds:
\begin{equation}\label{RepF}
[\cF h,D_s(\m,\cdot)x]_{\cD(s)}=[h,\cF(\m)^\zx
x]_{\cH_\D}=(\cF(\m)h,x)_{\sL_1\oplus\sL_2}
\end{equation}
for every $x\in\sL_1\oplus\sL_2$, $\m\in\Omega_s$.
\end{proof}
If the colligation $\D$ is simple, then the space $\cD$ is
isometrically isomorphic to the space $\cH$ under the mapping $\cF$.
If the colligation $\D$ is not simple but $\cH_\D$ is regular, then
the operator $\cF$ can be continued by zero to the subspace
$\cH\ominus\cH_\D$ and the continuation $\cF$ is given by the same
formula \eqref{Fur'e} for every $h\in\cH$. The operator $\cF$ is
called the {\it Fourier representation} of the colligation $\D$.

\begin{proposition}\label{Prop:FT}
(see \cite{KKhYu87} for the case $\k=0$) The Fourier representation
$\cF$ satisfies the relation
 \begin{equation}\label{E:FT}
  \cF P_\cH U^\zx+
  \begin{bmatrix}
   s(t)\\
   -I_{\sL_2}
  \end{bmatrix}
  P_{\sL_2} U^\zx=t\cdot\cF P_\cH+
  \begin{bmatrix}
   I_{\sL_1}\\
   -s(t)^*
  \end{bmatrix}
  P_{\sL_1}.
 \end{equation}
 Here $P_\cH$ and $P_{\sL_i}$ are orthoprojections onto $\cH$ and $\sL_i\ (i=1,2)$,
 respectively, $t\in\mathbb{T}$.
\end{proposition}
\begin{proof}
 Due to \eqref{harfun} and \eqref{G1G2} one can reduce the left-hand side of \eqref{E:FT}
 to the form
 $$
  \begin{bmatrix}
   P_{\sL_1}(I-t UP_\cH)^{-1}U(P_\cH U^\zx+P_{\sL_2}U^\zx)\\
   -\ov t P_{\sL_2}(I-\ov t U^\zx P_\cH)^{-1}U^\zx P_\cH U^\zx-P_{\sL_2}U^\zx
  \end{bmatrix}=
  \begin{bmatrix}
   P_{\sL_1}(I-t UP_\cH)^{-1}\\
   -P_{\sL_2}(I-\ov t U^\zx P_\cH)^{-1}U^\zx
  \end{bmatrix}.
 $$
 Similarly, the right-hand side of \eqref{E:FT} can be rewritten as
 \begin{equation*}
\begin{split}
  &\begin{bmatrix}
   t P_{\sL_1}(I-t UP_\cH)^{-1}P_\cH+P_{\sL_1}\\
   -P_{\sL_2}(I-\ov t U^\zx P_\cH)^{-1}U^\zx P_\cH-P_{\sL_2}(I-\ov t U^\zx P_\cH)^{-1}U^\zx
   P_{\sL_2}
  \end{bmatrix}\\
  &=
  \begin{bmatrix}
   P_{\sL_1}(I-t UP_\cH)^{-1}\\
   -P_{\sL_2}(I-\ov t U^\zx P_\cH)^{-1}U^\zx
  \end{bmatrix}.
  \end{split}
 \end{equation*}
 Now the equality \eqref{E:FT}  follows from two last equalities.
\end{proof}

\begin{definition}
The colligation $\Delta=(\sH,\sF,\sG;T,F,G,H)$ is called the
unitarily equivalent to the colligation
$\Delta^{'}=(\sH',\sF,\sG;T',F',G',H')$ if there exists a mapping
$Z$ from $\sH$ to $\sH'$ such that
\begin{equation}
T'=ZTZ^{-1}, \qquad F'=ZF, \qquad G'=GZ^{-1},
\end{equation}
or in other words
\begin{equation}
\begin{bmatrix}
  Z & 0 \\
  0 & I \\
\end{bmatrix}
\begin{bmatrix}
  T & F \\
  G & H \\
\end{bmatrix}
\begin{bmatrix}
  Z^{-1} & 0 \\
  0 & I \\
\end{bmatrix}=
\begin{bmatrix}
  T' & F' \\
  G' & H' \\
\end{bmatrix}
\end{equation}
\end{definition}
\begin{theorem}\label{fun uz}(see \cite{D2001})
Let $\Delta = (\cH,\sL_2,\sL_1;T,F,G,H)$ be a simple unitary
colligation and $s(\cdot)$ be the characteristic function of the
colligation $\D$, and let the colligation $\Delta_s =
(\cD(s),\sL_2,\sL_1;U_s)= (\cD(s),\sL_2,\sL_1;T_s,F_s,G_s,H_s)$,
where
\begin{equation}\label{Fcol}
 \begin{split}
  T_sf&=\ov t\sk({f-
  \begin{bmatrix}
   I&-s\\
   -s^*&I
  \end{bmatrix}
  \begin{bmatrix}
   f_+(0)\\
   0
  \end{bmatrix}
  }),\quad
  F_s=-\ov t
  \begin{bmatrix}
   I&-s\\
   -s^*&I
  \end{bmatrix}
  \begin{bmatrix}
   s(0)\\
   I_{\sL_2}
  \end{bmatrix},\\
  G_sf&=f_+(0),\quad H_s=s(0).
 \end{split}
\end{equation}

Then the colligations $\D$ and $\D_s$ are unitary equivalent via
\begin{equation}\label{uniteq}
 U_s
 \begin{bmatrix}
  \cF&0\\
  0&I_{\sL_2}
 \end{bmatrix}=
 \begin{bmatrix}
  \cF&0\\
  0&I_{\sL_1}
 \end{bmatrix}
 U
\end{equation}

\end{theorem}
\begin{proof}
The equality \eqref{E:FT} can be rewritten in the form
\begin{equation}\label{E:F1}
 \cF P_\cH+
 \begin{bmatrix}
  s\\
  -I_{\sL_2}
 \end{bmatrix}
 P_{\sL_2}=t\cdot \cF P_\cH U+
 \begin{bmatrix}
  I_{\sL_1}\\
  -s^*
 \end{bmatrix}
 P_{\sL_2}
\end{equation}
Hence one obtains for every $h\in\cH$ and $x\in\sL_2$
\begin{equation}\label{E:F2}
 \cF h=t\cdot \cF Th+
 \begin{bmatrix}
  I_{\sL_1}\\
  -s^*
 \end{bmatrix}
 Gh \quad (h\in\cH),
\end{equation}
\begin{equation}\label{E:F3}
 \begin{bmatrix}
  s\\
  -I_{\sL_2}
 \end{bmatrix}
 x=t\cdot \cF Fx+
 \begin{bmatrix}
  I_{\sL_1}\\
  -s^*
 \end{bmatrix}
 Hx \quad (x\in\sL_2).
\end{equation}
Let the operator $U_s=
\begin{bmatrix}
 T_s&F_s\\
 G_s&H_s
\end{bmatrix} $
be defined by the equality
$$
 U_s=
 \begin{bmatrix}
  \cF&0\\
  0&I_{\sL_1}
 \end{bmatrix}
 U
 \begin{bmatrix}
  \cF&0\\
  0&I_{\sL_2}
 \end{bmatrix}^{-1}
 =
 \begin{bmatrix}
  \cF T\cF^{-1}&\cF F\\
  G\cF^{-1}&H
 \end{bmatrix}.
$$
Setting $f=\cF h$, one obtains from \eqref{E:F2}, \eqref{E:F3} and
\eqref{ChF}
$$
T_sf=\cF T\cF^{-1}f=\ov t\sk({f-
\begin{bmatrix}
 I_{\sL_1}\\
 -s^*
\end{bmatrix}
G_s f })= \ov t\sk({f-
\begin{bmatrix}
 I_{\sL_1}\\
 -s^*
\end{bmatrix}
f_+(0) }),
$$
$$
F_sx=\cF Fx=\ov t\sk({
\begin{bmatrix}
 s\\
 -I_{\sL_2}
\end{bmatrix}
x-
\begin{bmatrix}
 I_{\sL_1}\\
 -s^*
\end{bmatrix}
H })= -\ov t
\begin{bmatrix}
 I_{\sL_1}&-s\\
 -s^*&I_{\sL_2}
\end{bmatrix}
\begin{bmatrix}
 s(0)x\\
 x
\end{bmatrix}
,
$$
$$
G_s f=G\cF^{-1}f=f_+(0),\quad H_s=H=s(0),
$$
which prove the formula \eqref{Fcol}.
\end{proof}

Theorem ~\eqref{fun uz} shows why the mapping $\cF$ is called the
Fourier representation of $\Delta$. In the case when the colligation
 $\Delta$ is simple, this mapping gives the unitary equivalence between $\Delta$ and its functional model.


\section{Abstract interpolation problem $AIP(\widetilde\k)$.}

 Given are  Hilbert spaces $\cH$, $\sL_1$, $\sL_2$, integer $\k,
\widetilde\k\in \mathbb{Z}_+$ and operators
$M,N\in\mathcal{B}(\cH)$, $C_1\in\mathcal{B}(\cH,\sL_1)$,
 $C_2\in\mathcal{B}(\cH,\sL_2)$, $P\in\mathcal{B}(\cH)$, such that

 (A1) $P=P^*$, $0\in\r(P)$ и $\sq_-(P)=\k$.

(A2) for every $f,g\in\cH$ the following identity holds
\begin{equation}\label{LSt}
(PMf,Mg)_\cH-(PNf,Ng)_\cH=(C_1f,C_1g)_{\sL_1}-(C_2f,C_2g)_{\sL_2}
\end{equation}

Find an operator function $s(\cdot)\in
S_{\widetilde\k}(\sL_2,\sL_1)$ and the mapping $\Phi:\cH\to\cD(s)$,
such that:
\begin{itemize}
\item[(i)]
$[\Phi h,\Phi h]_{\cD(s)}\leq(Ph,h)_\cH,\quad \text{ for every
}h\in\cH; $
\item[(ii)]
$\Phi Mh-t\Phi Nh=
\begin{bmatrix}
I&-s \\
-s^*&I
\end{bmatrix}
Ch$, where $C=\begin{bmatrix}
 C_1\\
C_2
\end{bmatrix}$,
$t\in\dT$, for every $h\in\cH$.
\end{itemize}

 Alongside with the Problem $AIP(\w\k)$
let us consider the Problem $AIP_0(\w\k)$, by replacing $(i)$ by the
generalized Parseval equality
\begin{itemize}
\item[(ii')]
$[\Phi h,\Phi h]_{\cD(s)}=(Ph,h)_\cH$, for every $h\in\cH$.
\end{itemize}
Let $\cH$ be supplied with the inner product
$[\cdot,\cdot]_\cH:=(P\cdot,\cdot)_\cH$. Then the space
$(\cH,[\cdot,\cdot]_\cH)$ is a Pontryagin space with the negative
index $\k$.

It follows from the identity \eqref{LSt} that the operator
\begin{equation}\label{E:V}
 V:\begin{bmatrix}
 Mf\\
C_2f
\end{bmatrix}\to
\begin{bmatrix}
Nf \\
C_1f
\end{bmatrix},
\end{equation}
 is a Pontryagin space isometric operator from $\cH\oplus\sL_2$ to $\cH\oplus\sL_1$.
 The problem $AIP(\w\k)$ can be reduced to the problem of extension of the isometric operator $V$
  to a unitary operator
$$
U=
\begin{bmatrix}
 T&F\\
G&H
\end{bmatrix}
:\begin{bmatrix}
 \w\cH\\
\sL_2
\end{bmatrix}\to
\begin{bmatrix}
 \w\cH\\
\sL_1
\end{bmatrix},\quad(\w\cH\supset\cH).
$$
\begin{definition}\label{Def:reg}
A unitary extension $U$ of $V$ will be called
$(\sL_2,\sL_1)$-regular, if $\w\cH\ominus\cH_\Delta$ is a Hilbert
space. An extension $U$ will be called $(\sL_2,\sL_1)$-minimal, if
the corresponding colligation $\Delta=(\w\cH,\sL_2,\sL_1, U)$ is
 simple, or in other words $\cH_\D=\w\cH$. Clearly, every $(\sL_2,\sL_1)$-minimal unitary extension $U$ is $(\sL_2,\sL_1)$-regular.

\end{definition}

In the case $\w\k=0$ a description of the set of solutions of the
Problem $AIP(\w\k)$ was given in \cite{KKhYu87}, \cite{KhYu94}.
\begin{lemma}\label{L:reg}
Let  $U$ be a unitary operator in a Pontryagin space $\w\cH$ with
the negative index $\w\k$ and $
s(\l)=P_{\sL_1}(I-zUP_{\w\cH})^{-1}UP_{\sL_2}$ be its characteristic
function. Then $s(\cdot)$ belongs $S_{\w\k}(\sL_2,\sL_1)$ if and
only if the operator $U$ is a $(\sL_2,\sL_1)$-regular.
\end{lemma}
\begin{proof}
If $U$ is a $(\sL_2,\sL_1)$-regular, then $\ind_-(\cH_\Delta)=\w\k$.
The mapping
\begin{equation}
\cF(\l):h\to
\begin{bmatrix}
 P_{\sL_1}(I-\l UP_{\w\cH})^{-1}UP_\cH\\
-\ov\l P_{\sL_2}(I-\ov\l U^\zx P_{\w\cH})^{-1}U^\zx P_\cH
\end{bmatrix}h
\end{equation}
is isometric from $\cH_\Delta$ to $\cD(s)$. Hence,
$\ind_-(\cD(s))=\w\k$ and $s(\cdot)\in S_{\w\k}(\sL_2,\sL_1)$.

Conversely, if  $s\in S_{\w\k}$, then since $\ind_-(\cD(s))=\w\k$
 one gets $\ind_-(\cH_\Delta)=\w\k$. Hence, $U$ is a
$(\sL_2,\sL_1)$-regular.
\end{proof}
%

\begin{theorem}\label{sAIP}
For the Problem $AIP(\w\k)$ to be solvable it is necessary that
$\k\leq\w\k\in \mathbb{Z}_+$. The formulas
\begin{equation}\label{A7}
s(\l)=P_{\sL_1}(I-\l UP_{\w\cH})^{-1}UP_{\sL_2}
\end{equation}
\begin{equation}\label{A8}
\Phi(\l)=
\begin{bmatrix}
 P_{\sL_1}(I-\l UP_{\w\cH})^{-1}UP_\cH\\
-\ov\l P_{\sL_2}(I-\ov\l U^\zx P_{\w\cH})^{-1}U^\zx P_\cH
\end{bmatrix}
\end{equation}
establish a one-to-one correspondence between the set of solutions
$\{s,\Phi\}$ of the Problem $AIP(\w\k)$ and the set of all
$(\sL_2,\sL_1)$-regular unitary operator extensions $U$ of the
operator $V$, such that:
\begin{equation}\label{A9}
ind_-(\w\cH)=\w\k.
\end{equation}
A solution $\{s,\Phi\}$ of the Problem $AIP(\w\k)$ is a solution of
the Problem $AIP_0(\w\k)$, if and only if the extension $U$ is
$(\sL_2,\sL_1)$-minimal.
\end{theorem}
\begin{proof}
1) Let $U$ be a $(\sL_2,\sL_1)$-regular extension of the operator
$V$, satisfying ~\eqref{A9} and let $P_{\Delta}$ be the orthogonal
projection onto $\cH_{\Delta}$ in $\w \cH$. Consider a unitary
colligation
$$\Delta=(\w\cH,\sL_2,\sL_1;U),\quad U=\begin{bmatrix}
T&F\\
G&H
\end{bmatrix}.$$
Due to Proposition \ref{PropFM1} the space $\cD(s)$ can be
interpreted as the set of vector function
$$
(\cF h)(\l)=\begin{bmatrix}
 G(I-\l T)^{-1}h \\
 -\ov\l F^\zx (I-\ov\l T^\zx)^{-1}h
\end{bmatrix},\quad h\in\w\cH
$$
with the scalar product
$$
[\cF h,\cF g]_{\cD(s)}=[P_\Delta h,P_\Delta g]_{\w\cH}.
$$
Since $\w\cH\ominus\cH_\Delta$ is a Hilbert space for every
$h\in\w\cH$ the following inequality holds
$$
[\cF h,\cF h]_{\cD(s)}=[P_\Delta h,P_\Delta
h]_{\w\cH}\leq[h,h]_{\w\cH}.
$$
Setting $\Phi h=\cF h$ for $h\in\cH$, one obtains the mapping
$\Phi:h\to\cD(s)$, which satisfies (i).

 The equality (ii)
is implied by relation \eqref{E:FT}
\begin{equation}\label{A12}
\cF P_{\w\cH}U^\zx+\begin{bmatrix}
 s\\
-I_{\sL_2}
\end{bmatrix}P_{\sL_2}U^\zx=t\cF P_{\w\cH}+\begin{bmatrix}
 I_{\sL_1}\\
 -s^*
\end{bmatrix}P_{\sL_1}
\end{equation}
For a vector $\begin{bmatrix}
 Mh\\
C_2h
\end{bmatrix}\quad (h\in\cH)$
from $\dom V$ one has \eqref{E:V}
\begin{equation}\label{A13}
U^\zx\begin{bmatrix}
 Nh\\
C_1h
\end{bmatrix}
=
\begin{bmatrix}
 Mh\\
C_2h
\end{bmatrix}\quad (h\in\cH)
\end{equation}
Substituting ~\eqref{A13} into ~\eqref{A12} and taking account of
$\cF Mh=\Phi Mh$, $\cF Nh=\Phi Nh$, one obtains the equality
$$
\Phi Mh+\begin{bmatrix}
 s\\
-I_{\sL_2}
\end{bmatrix}C_2h=t\Phi N h+\begin{bmatrix}
 I_{\sL_1}\\
- s^*
\end{bmatrix}C_1 h,
$$
which is equivalent (ii).

2) Conversely, let $\{s,\Phi\}$ be a solution of the Problem
$AIP(\w\k)$ and let $\Delta_s=(\cD(s),\sL_2,\sL_1;U_s)$ be a unitary
colligation with the characteristic function $s(\cdot)$, built in
Proposition ~\ref{fun uz}. Since the operator $I-\Phi^\zx
\Phi:\cH\to\cH$ is nonnegative in the Pontryagin space $\cH$, it
admits a Bognar-Kramli factorization \cite{Bog74}: $I-\Phi^\zx
\Phi=DD^\zx$, where the defect operator $D$ acts from the Hilbert
space $\cD=\ov\ran(I-\Phi^\zx\Phi)$ to the Pontryagin space $\cH$.
 Let us construct a lifting  $\w\Phi:\cH\to\cD(s)\oplus\cD=:\w\cH$
 of the mapping
$\Phi:\cH\to\cD(s)$, setting $\w\Phi:=\begin{bmatrix}
\Phi \\
D^\zx
\end{bmatrix}$. Then for every $h,g\in\cH$ one obtains the equality:
$$
[\w\Phi h,\w\Phi g]_{\w\cH}=[\Phi f,\Phi g]_{\cD(s)}+[D^\zx h,D^\zx
g]_\cD=(P h,g)_\cH,
$$
which proves that the mapping $\w\Phi$ is isometric. Further it
follows from ~\eqref{uniteq} that
\begin{equation}\label{A14}
U_s
\begin{bmatrix}
\Phi M h \\
C_2h
\end{bmatrix}
=\begin{bmatrix}
 \cF&0\\
 0&I_{\sL_1}
\end{bmatrix}
U
\begin{bmatrix}
 \cF^{-1}&0\\
 0&I_{\sL_2}
\end{bmatrix}
\begin{bmatrix}
\Phi M h \\
C_2h
\end{bmatrix}
 =\begin{bmatrix}
\Phi N h \\
C_1h
\end{bmatrix},\quad
h\in\cH
\end{equation}
Since the operators $U_s$, $\w\Phi$ and $V$ are isometric, one
obtains for every $h\in\cH$:
\begin{multline*}
(D^\zx Nh,D^\zx Nh)_\cH=[\w\Phi Nh,\w\Phi Nh]_{\w\cH}-[\Phi Nh,\Phi
Nh]_{\cD(s)}\\
=[\w\Phi Nh,\w\Phi Nh]_{\w\cH}-[\Phi Mh,\Phi
Mh]_{\cD(s)}-(C_2h,C_2h)_{\sL_2}+(C_1h,C_1h)_{\sL_1}\\
=[(PNh,Nh)_\cH+\|C_1h\|^2_{\sL_1}]-[(PMh,Mh)_\cH+\|C_2h\|^2_{\sL_2}]\\
+[D^\zx Mh,D^\zx Mh]_\cH=[D^\zx Mh,D^\zx Mh]_\cH.
\end{multline*}
Thus the operator
\begin{equation}\label{A15}
U_D:D^\zx Mh\to D^\zx Nh,\quad h\in\cH
\end{equation}
is isometric. Let $\w U_D$ be a unitary extension of $U_D$ in a
Hilbert space $\w\cD\supset\cD$. Then
\begin{equation*}
U=U_s\oplus\w U_D:\begin{bmatrix}
\w\cH \\
\sL_2
\end{bmatrix}
\to
\begin{bmatrix}
\w\cH \\
\sL_1
\end{bmatrix},\quad\w\cH=\cD(s)\oplus\w\cD
\end{equation*}
is a unitary operator. It follows from \eqref{A14} and \eqref{A15}
that
\begin{equation*}
U\begin{bmatrix}
 \w\Phi Mh\\
C_2h
\end{bmatrix}=
\begin{bmatrix}
 \w\Phi Nh\\
C_1h
\end{bmatrix},\quad h\in\cH
\end{equation*}
and the operator $U$ is a unitary extension of the isometric
operator
\begin{equation*}
\w V=\begin{bmatrix}
 \w\Phi&0\\
0&I
\end{bmatrix}V
\begin{bmatrix}
 \w\Phi&0\\
0&I
\end{bmatrix}^{-1}
\end{equation*}
Hence, $\k\leq\w\k$ and $U$ is an $(\sL_2,\sL_1)$-regular extension
of the operator $\w V$ satistying \eqref{A9}, since
${\w\cH\ominus\cD(s)=\w\cD}$ is a Hilbert space.

3) If the extension $U$ is $(\sL_2,\sL_1)$-minimal, then
$\cH_\Delta=\w\cH$ and the mapping $\cF:\w\cH\to\cD(s)$ is
isometric. It proves the Parseval equality (ii'). Conversely, if
$\{s,\Phi\}$ is a solution of the Problem $AIP_0(\w\k)$, then the
mapping $\Phi$ is isometric and the operator $U_s$ is a unitary
extension of the operator $\w V=\begin{bmatrix}
 \Phi&0\\
0&I
\end{bmatrix}V
\begin{bmatrix}
 \Phi&0\\
0&I
\end{bmatrix}^{-1}$. Since the colligation $\Delta_s$ is simple, the extension $U$ is $(\sL_2,\sL_1)$-minimal.
 \end{proof}

\section{Parametrization of solutions.}

\begin{definition}
Recall that $\l$ is a regular point of the pencil $M-\l N$ (it is
denoted by $\l\in\r(M,N)$ ) if $0\in\r(M-\l N)$. Denote
$$
\r(M,N)^{\#}:=\sk\{{\l:0\in\r\sk({M-\frac{1}{\ov\l} N})^*}\}
$$

\end{definition}
Suppose that the Problem $AIP(\w\k)$ data satisfy the condition


(A3) There exists a point $a\in\mathbb{T}\bigcap\r(M,N)$.

We will suppose that $\r(M,N)\supset\mathbb{D}$ except finite set of
points.

 Recall the following definition from the paper \cite{B2}.
\begin{definition}
We shall write $\l\in\r( V,\sL_2)$ if $1\in\wh \r(\l P_\cH V)$ and
\begin{equation}\label{lreg}
(I-\l P_{\cH} V)\dom V\dotplus\begin{bmatrix}
 0\\
\sL_2 \end{bmatrix}=\begin{bmatrix}
 \cH\\
\sL_2
\end{bmatrix},
\end{equation}
and $\l\in\r( V^{-1},\sL_1)$ if $1\in\wh \r(\l P_\cH V^{-1})$ and
\begin{equation}\label{lreg2}
(I-\l P_{\cH} V^{-1})\ran V\dotplus\begin{bmatrix}
 0\\
\sL_1
\end{bmatrix}=\begin{bmatrix}
 \cH\\
\sL_1
\end{bmatrix}.
\end{equation}
We shall write $\l\in\r_{V}(\sL_2,\sL_1)$ if $\l\in\r( V,\sL_2)$,
and $\overline{\l}\in\r( V^{-1},\sL_1)$.
\end{definition}

\begin{proposition}
The following equivalences hold

1) $\l\in\r(V,\sL_2)$ iff $\l\in\r(M,N)$;

2) $\ov\l\in\r(V^{-1},\sL_1)$ iff $\ov\l\in\r(N,M)$;

3) $\l\in\r_V(\sL_2,\sL_1)$ iff $\l\in\r(M,N)\bigcap\r(N,M)^*$.
\end{proposition}
\begin{proof}
$1)$ The statement $\l\in\r(V,\sL_2)$ means that for every vector $
 \begin{bmatrix}
  f\\
  u_2
 \end{bmatrix}
 \in
 \begin{bmatrix}
  \cH\\
  \sL_2
 \end{bmatrix}
$ there exist uniquely determined vectors $h\in \cH$ and
$l_2\in\sL_2$ such that
$$
 \begin{bmatrix}
 f\\
 u_2
\end{bmatrix}=
(I-\l P_\cH V)
\begin{bmatrix}
 Mh\\
 C_2h
\end{bmatrix}+
\begin{bmatrix}
 0\\
 l_2
\end{bmatrix},
$$
$$
\begin{bmatrix}
 f\\
 u_2
\end{bmatrix}=
\begin{bmatrix}
 (M-\l N)h\\
 C_2h+l_2
\end{bmatrix}.
$$
The condition of the unique representation for every vector $
\begin{bmatrix}
 f\\
 u_2
\end{bmatrix}
$ gives an invertibility of $M-\l N$.

Conversely, let $0\in\r(M-\l N)$, then for every $
 \begin{bmatrix}
  f\\
  u_2
 \end{bmatrix}
 \in
 \begin{bmatrix}
  \cH\\
  \sL_2
 \end{bmatrix}
$ one can define vectors $h:=(M-\l N)^{-1}f$ and $l_2:=u_2-C_2(M-\l
N)^{-1}f$. The vector $
 \begin{bmatrix}
  f\\
  u_2
 \end{bmatrix}$ can be represented in the following way

$$
 \begin{bmatrix}
 f\\
 u_2
\end{bmatrix}=
(I-\l P_\cH V)
\begin{bmatrix}
 M(M-\l N)^{-1}f\\
 C_2(M-\l N)^{-1}f
\end{bmatrix}+
\begin{bmatrix}
 0\\
 u_2-C_2(M-\l N)^{-1}f
\end{bmatrix}.
$$
$2)$ Let the statement $\ov \l\in\r(V^{-1},\sL_1)$ hold. This means
that for every vector $
 \begin{bmatrix}
  f\\
  u_1
 \end{bmatrix}
 \in
 \begin{bmatrix}
  \cH\\
  \sL_1
 \end{bmatrix}
$ there exists uniquely determined vectors $h\in \cH$ and
$l_1\in\sL_1$, such that
$$
 \begin{bmatrix}
 f\\
 u_1
\end{bmatrix}=
(I-\ov \l P_\cH V^{-1})
\begin{bmatrix}
 Nh\\
 C_1h
\end{bmatrix}+
\begin{bmatrix}
 0\\
 l_1
\end{bmatrix},
$$
$$
\begin{bmatrix}
 f\\
 u_1
\end{bmatrix}=
\begin{bmatrix}
 (N-\ov \l M)h\\
 C_1h+l_1
\end{bmatrix}.
$$
Therefore, the operator $N-\ov \l M$ is invertible (i.e.
$\ov\l\in\r(N,M)$).

The converse proposition is proved in the way similar to the
converse proposition of 1).

The proposition 3) follows from 1) and 2).
\end{proof}

\begin{corollary}
For $a\in \mathbb{T}$ we have
$\r(V,\sL_2)=\r(V^{-1},\sL_1)=\r_V(\sL_2,\sL_1)$.

\end{corollary}

\begin{definition}
For $\l\in \r(V,\sL_2)$, by $\Pp(\l)$ we denote the skew projection
onto $\sL_2$ in the decomposition (\ref{lreg}) and introduce the
operator
$$
\cQ_{\sL_1}(\l):=P_{\sL_1} V(I-\l P_{\cH} V)^{-1}(I-\Pp(\l)).
$$
\end{definition}

Introduce the operator-function $W(\l)$ defined by
\begin{equation}\label{W}
J- W(\l)J W(\m)^*=(1-\l\ov\m)\cG(\l)\cG(\m)^\zx,
\end{equation}
where
\begin{equation}\label{JG}
J=\begin{bmatrix} I_{\sL_1}&0 \\0&-I_{\sL_2}
\end{bmatrix}, \quad\mathcal{G}(\l):=\begin{bmatrix} \cQ_{\sL_1}(\l) \\I_{\sL_2}-\Pp(\l)
\end{bmatrix}:\begin{bmatrix}\cH  \\\sL_2
\end{bmatrix}\to\begin{bmatrix} \sL_1 \\\sL_2
\end{bmatrix}.
\end{equation}

\begin{definition}
The operator-function $$W(\l)=\begin{bmatrix}  w_{11}(\l)& w_{12}(\l) \\
w_{21}(\l)& w_{22}(\l)\end{bmatrix}:\begin{bmatrix} \sL_1 \\\sL_2
\end{bmatrix}\rightarrow\begin{bmatrix}  \sL_1 \\\sL_2
\end{bmatrix}, (\l\in\r_V(\sL_2,\sL_1)),$$ satisfying the equality (\ref{W}) is called the resolvent matrix for the operator $V$.
\end{definition}
\begin{proposition}
If the assumptions (A1)-(A3) hold, then the
 resolvent matrix $W(\cdot)$ can be defined by the equality
\begin{equation}\label{WW}
W(\l)=I-(1-\l\ov a)\cG(\l)\cG(a)^\zx
J_\sL,\quad\l\in\r_V(\sL_2,\sL_1).
\end{equation}
\end{proposition}

 Let us find the explicit form of the operators
 $\Pp(\l)$, $\cQ_{\sL_1}(\l)$ and $W(\l)$. Let
$u,v\in\sL_2$ and $f,h\in \cH$, then
\begin{equation*}
\begin{bmatrix}
 0\\
u
\end{bmatrix}+(I-\l P_\cH V)\begin{bmatrix}
 Mh\\
C_2 h
\end{bmatrix}=\begin{bmatrix}
 (M-\l N)h\\
C_2 h+u
\end{bmatrix}=
\begin{bmatrix}
 f\\
v
\end{bmatrix}
\end{equation*}
Hence, the operator $\Pp(\l)$ is
\begin{equation}\label{BP}
\Pp(\l)\begin{bmatrix}
 f\\
v
\end{bmatrix}=
\begin{bmatrix}
 0&0\\
-C_2(M-\l N)^{-1}&I
\end{bmatrix}
\begin{bmatrix}
 f\\
v
\end{bmatrix}
\end{equation}
and the adjoint operator is
\begin{equation}\label{BPs}
\Pp(\l)^\zx\begin{bmatrix}
 0\\
u
\end{bmatrix}=
\begin{bmatrix}
0& -P^{-1}(M-\l N)^{-*}C_2^*\\
0&I
\end{bmatrix}
\begin{bmatrix}
 0\\
u
\end{bmatrix},\quad u\in\sL_2.
\end{equation}
Let us find the explicit form for $\cQ_{\sL_1}(\l)=P_{\sL_1}V(I-\l
P_\cH V)^{-1}(I-\Pp(\l))$:
$$
(I-\Pp(\l))\begin{bmatrix}
f \\
v
\end{bmatrix}=
\begin{bmatrix}
f \\
C_2(M-\l N)^{-1}f
\end{bmatrix}=(I-\l P_\cH V)\begin{bmatrix}
 Mh\\
C_2h
\end{bmatrix},
$$
hence, $h=(M-\l N)^{-1}f$.
$$
\cQ_{\sL_1}(\l)
\begin{bmatrix}
f \\
v
\end{bmatrix}=P_{\sL_1}V
\begin{bmatrix}
M h \\
C_2h
\end{bmatrix}=P_{\sL_1}
\begin{bmatrix}
N h \\
C_1h
\end{bmatrix}=C_1(M-\l N)^{-1}f,
$$
Therefore, the $\cQ_{\sL_1}(\l)$ is
\begin{equation}\label{BQ}
\cQ_{\sL_1}(\l)\begin{bmatrix}
 f\\
v
\end{bmatrix}=
\begin{bmatrix}
 0&0\\
C_1(M-\l N)^{-1}&0
\end{bmatrix}
\begin{bmatrix}
 f\\
v
\end{bmatrix}
\end{equation}
and the adjoint operator is
\begin{equation}\label{BQs}
\cQ_{\sL_1}(\l)^\zx\begin{bmatrix}
 0\\
v
\end{bmatrix}=
\begin{bmatrix}
0& P^{-1}(M-\l N)^{-*}C_1^*\\
0&0
\end{bmatrix}
\begin{bmatrix}
 0\\
v
\end{bmatrix},\quad v\in\sL_1.
\end{equation}
Then from \eqref{WW} we get the explicit form for the resolvent
matrix $W(\l)$:
\begin{equation}\label{BW}
W(\l)=I-(1-a\l)C(M-\l N)^{-1}P^{-1}(M-\ov aN)^{-*}C^*J,\quad a\in
\mathbb{T},
\end{equation}
where $J$ is defined by \eqref{JG}.

Consider the main properties of $W(\cdot)$. We will need one more
condition

(A4) The set of points $\mathbb{D}\setminus\r(M,N)$ consists of at
most of countable set of isolated points.
\begin{proposition}\label{Prop:Winp}
If the assumptions (A1)-(A4) are in force, then
$W(\cdot)\in\mathcal{P}_{\k'}(J)$ for some $\k'\leq\k$.
\end{proposition}
\begin{proof}
This follows from the equality \eqref{W}. Indeed, for some
$n\in\mathbb{N}$ one get
\begin{equation}\label{kvf}
\begin{split}
&\sum_{i,j=1}^n\sk({K_{\omega_j}(\omega_i)h_i,h_j})_{\sL_1\oplus\sL_2}\xi_i\ov\xi_j=
\sum_{i,j=1}^n\sk({\cG(\omega_i)\cG(\omega_j)^\zx h_i,h_j})_{\sL_1\oplus\sL_2}\xi_i\ov\xi_j\\
&=\sk[{\sum_{i=1}^n\cG(\omega_i)^\zx
h_i\xi_i,\sum_{j=1}^n\cG(\omega_j)^\zx h_j\xi_j}]
\end{split}
\end{equation}
Hence, the quadratic form has at most $\k$ negative squares.
\end{proof}
\begin{proposition}\label{Prop:WinP}
Let the assumptions (A1)-(A4) hold. Then
 $W(\cdot)\in\mathcal{P}_{\k}(J)$ if and only if
\begin{equation}\label{kerpen}
\bigcap_{\l\in\r(M,N)}\ker C(M-\l N)^{-1}=\{0\}.
\end{equation}
\end{proposition}

\begin{proof}
Notice that the condition \eqref{kerpen} holds if and only if
$$
\span\sk\{{\{\cG(\omega_j)^\zx
h_j\}:\omega_j\in\r(M,N),h_j\in\sL_1\oplus\sL_2}\}
$$
dense in $\cH\oplus\sL_2$. Indeed, if $h[\p]\cG(\omega)^\zx u$,
where $\omega\in\r(M,N)$ and $u\in\sL_1\oplus\sL_2$ then
\begin{equation}
\begin{split}
&\sk({C(M-\omega N)^{-1}h,u})_{\sL_1\oplus\sL_2}=\sk({h,(M^*-\ov
\omega N^*)^{-1}C^*u})_{\cH\oplus\sL_2}\\&=\sk({Ph,P^{-1}(M^*-\ov
\omega N^*)^{-1}C^*u})_{\cH\oplus\sL_2}=\sk[{h,\cG(\omega)^\zx
u}]_{\cH\oplus\sL_2}=0.
\end{split}
\end{equation}
Therefore, $C(M-\omega N)^{-1}h=0$. To complete the prove we need to
use the previous Proposition \ref{Prop:Winp}.
\end{proof}
\begin{proposition}\label{Prop:WinU}
Let the assumptions (A1)-(A4) are in force and \eqref{kerpen} hold.
If in addition
\begin{equation}\label{mu pen}
m(\sigma(M,N)\cap\mathbb{T})=0,
\end{equation}
where $m(\cdot)$ is a Lebesgue measure, then
$W(\cdot)\in\mathcal{U}_{\k}(J)$.
\end{proposition}
\begin{proof}
This follows from \eqref{W}. Hence, for $\m\in\r(M,N)\cap\mathbb{T}$
we get
$$
J- W(\m)J W(\m)^*=(1-|\m|^2)\cG(\m)\cG(\m)^\zx=0.
$$
In other words, $W(\cdot)$ is a $J$-unitary operator-function for
almost all $\m\in\mathbb{T}$. Using Proposition \ref{Prop:WinP} one
has $W(\cdot)\in\mathcal{U}_{\k}(J)$.
\end{proof}

The matrix-function $s(\l)$ is a component of the solution of the
Problem $AIP(\w\k)$. According to Theorem \ref{sAIP} this
matrix-function is the characteristic function of the unitary
colligation $\Delta=(\cH,\sL_2,\sL_1;U)$, where the operator $U$ is
a unitary extension of an isometric operator $V$. A description of
characteristic functions of unitary colligations was obtained in
\cite{B2}.

The following Theorem gives a description of solutions.
\begin{theorem}\label{s(l)}
Let the data of the Problem $AIP(\k)$ satisfies the assumptions
(A1)-(A3). The the solution set of $AIP(\k)$ is described by the
formula
\begin{equation}\label{sW}
s(\l)=(w_{11}(\l)\e(\l)+w_{12}(\l))(w_{21}(\l)\e(\l)+w_{22}(\l))^{-1},
\end{equation}
where $\e(\cdot)$ ranges over the class $S(\sL_2,\sL_1)$ and
$w_{21}(0)\e(0)+w_{22}(0)$ is invertible. In this case the mapping
$\Phi:\cH\to\cD(s)$ is uniquely defined by
$$
\Phi(t)=
\begin{bmatrix}
 I&-s(t)\\
 -s^*(t)&I
\end{bmatrix}
C(M-tN)^{-1},\ t\in\dT.
$$
\end{theorem}

\begin{proof}
Let $\{s,\Phi\}$ be a solution of the Problem $AIP(\k)$, where
$s(\cdot)\in S_{\k}(\sL_2,\sL_1)$ and is holomorphic on a
neighborhood of $0$. From Theorem \ref{sAIP} we get
\begin{equation}\label{E:th3.1}
s(\l)=P_{\sL_1}(I-\l UP_{\w\cH})^{-1}UP_{\sL_2}.
\end{equation}
Further, using Theorem 3 from \cite{B2}, we get
$s(\l)=T_{W(\l)}[\e(\l)]$, where $\e(\cdot)\in S(\sL_2,\sL_1)$, and
 $w_{21}(\cdot)\e(\cdot)+w_{22}(\cdot)$ is invertible at $0$.

 Conversely, let $\e(\cdot)\in S(\sL_2,\sL_1)$, $s(\cdot)=T_W[\e]\in S_{\k}(\sL_2,\sL_1)$
 and the matrix-function be holomorphic at $0$. According to Theorem 3 in \cite{B2}
 $s(\cdot)$ admits the representation \eqref{E:th3.1}. Since $s(\cdot)\in S_{\k}(\sL_2,\sL_1)$ (see Lemma
  \ref{L:reg}), the unitary operator $U$
 is a $(\sL_2,\sL_1)$-regular.
\end{proof}




\newpage
\begin{thebibliography}{98}
\addcontentsline{toc}{section}{\bf СПИСОК ЛИТЕРАТУРЫ}
{\normalsize
\bibitem{ADRS91}
        Alpay D., Dijksma A., Rovnyak J., and de~Snoo H.S.V.
        Schur functions, operator colligations, and reproducing kernel
        Pontryagin spaces, Oper. Theory: Adv. Appl., {\bf 96},
        Birkh\"auser Verlag, Basel, 1997.

\bibitem{AD86}
        Alpay D. and Dym H. On applications of reproducing
        kernel spaces to the Schur algorithm and rational J unitary
        factrization. I. Schur methods in operator theory and signal
        processing, pp. 89-159 in: Oper. Theory Adv. Appl., 18, Birkhauser
        Verlag, Basel, 1986.

\bibitem{A50}
        Aronszajn N. Theory of reproducing kernels, Trans. Amer.
        Math. Soc., 68 (1950), 337-404.
\bibitem{AG92}
        Arov D.Z., Grossman L.Z. Scattering matrices in the theory
        of unitary extensions of isometric operatirs, Math. Nachr.,
        157 (1992), 105-123.
\bibitem{AI86}
        Azizov ~T.~Ya. and Iokhvidov ~I.~S. Foundations
        of the theory of linear operators in spaces with an indefinite
        metric, Moscow,
        Nauka. 1986, (English translation: Wiley, New York, 1989).
\bibitem{B2}
        Baidiuk D. Description of Scattering Matrices of Unitary
        Extensions of Isometric Operators in Pontryagin Space, Math.
        Notes 73 (2013), no.6, 940-943.
\bibitem{Ben72}
        Benewitz C. Symmetric relations on a Hilbert space, Lect.Notes Math.
        280 (1972), 212-218.


\bibitem{Bog74}
        Bognar J. Indefinite inner product spaces, Ergeb. Math.
        Grenzgeb.,Bd 78 Springer Verlag, New York-Heidelberg, 1974.
\bibitem{dBrR66}
        De~Branges ~L. and Rovnyak ~J. Canonical models in quantum scattering theory, Perturbation Theory
        and its Application in Quantum Mechanics, Wiley, New York, (1966), 359-391.

\bibitem{Br78}
        Brodskii M.S. Unitary operator colligations and their
        characteristic functions, Uspekhi Mat.Nauk, 33 (1978), no. 4,
        141-168; Engl. trans.: Rus. Math. Surveys, 33 (1978), no. 4,
        159-191.
\bibitem{BrGK70}
         Brodskii V. M., Gokhberg I. Ts., and Krein M. G. Funkts. Anal. Prilozhen., 4 (1970), no. 1, 88–90.



\bibitem{D2001}
        Derkach V. A.
        On indefinite abstract interpolation problem,
        Methods of Funct. Analysis and Topology, 7 (2001), no. 4, 87-100.

\bibitem{DD09}
        Derkach V., Dym H. On Linear Fractional
        Transformation s Associated with Generalized $J$-Inner Matrix
        Functions, Integr.Equ.Oper. Theory 65 (2009), 1-50.


\bibitem{DHMS04}
        Derkach V.A.,  Hassi S., Malamud M.M., and de Snoo H.S.V. Boundary
        relations and their Weyl families, Transactions of AMS, 358 (2004),
        no. 12, 5351-5400.

\bibitem{DM95}
        Derkach V.A. and Malamud M.M. The extension theory of
        hermitian operators and the moment problem, J. Math. Sci.,
        73 (1995), no. 2. 141-242.

\bibitem{DLS}
        Dijksma A., Langer H. and  de~Snoo H.S.V.
        Unitary colligations in Kre\u{\i}n spaces and their role in the
        extension theory of isometries and symmetric linear
        relations, Lecture Notes in
        Mathematics, 1242 (1987), 1-42.




\bibitem{G88}
        Gantmacher F.R.  Theory of Matrices Nauka, Moscow, 1988 (in Russian)




\bibitem{KKhYu87}
        Katsnelson V.E., Kheifets A.Ya. and Yuditskii P.M. The
        abstract interpolation problem and extension theory of
        isomretric operators, Operators in Spaces of Functions
        and Problems in Function Theory, Kiev, Naukova Dumka, (1987),
        83-96 (Russian).

\bibitem{Kh90}
        Kheifets A.Ya. Generalized bi-tangential
        Schur-Nevanlinna-Pick problem and related Parseval equality,
        Teor. Funkt. Anal. i ih Prilozhen., Kharkov, 54 (1990), 89-96
        (Russian).

\bibitem{Kh96}
        Kheifets A.Ya. Hamburger moment problem: Parseval equality
        and A-singularity, J. Funct. Analysis, 141 (1996), 374-420.

\bibitem{KhYu94}
        Kheifets A.Ya. and Yuditskii P.M. An analysis and extension of V.P.
        Potapov's approch to interpolation problems with applications to the
        generalized bi-tangential Schur-Nevanlinna-Pick problem and
        J-inner-outer factorization, Operator
        Theory:Adv.Appl., Birkhauser, Basel, 72 (1994), 133-161.

\bibitem{KP74}
        Kovalishina I.V. and Potapov V.P. Indefinite metric in
        Nevanlinna-Pick problem, Dokl. Akad. Nauk. Armjan. SSR, Ser.
        mat., 9 (1974), no. 1, 3-9.

\bibitem{Kr44}
        Kre\u{\i}n M.G. On Hermitian operators whith defect indices
        equal to one, Dokl. Akad. Nauk SSSR, 43 (1944), no. 8,
        339-342.

\bibitem{KrL72}
        Kre\u{\i}n M.G. and Langer H.
        , Uber die verallgemeinerten Resolventen und die characteristische
        Function eines isometrischen Operators im Raume
        $\Pi_\kappa$,
        Hilbert space Operators and Operator Algebras
        (Proc.Intern.Conf.,Tihany, 1970 );
         Colloq.Math.Soc.Janos Bolyai, North--Holland, Amsterdam, 5 (1972), 353-399.


\bibitem{KrS66}
        Kre\u{\i}n M.G. and Saakyan Sh.N. Some new results in the
        theory of resolvent matrices of Hermitian operators, Dokl. Akad. Nauk SSSR, 169 (1966), no. 1,
        657-660.

\bibitem{Kuzh96}
        Kuzhel' S.A. Abstract scattering scheme of Lax and Phillips in Pontryagin spaces,
        SB MATH, 187 (1996), no. 10, 1503–1523.

\bibitem{L46}
         Liv\v{s}ic M.S. On a certain class of linear operators on Hilbert spaces Mat. Sb.
         (N.S.), 19 (1946), no. 61, 232–260 (in Russian)
\bibitem{P55}
        Potapov V.P. Multiplicative structure of $J$-nonexpamding
        matrix functions, Trudy Mosk. Matem. Obsch., 4 (1955),
        125-236.
\bibitem{Sch64}
        L. Schwartz Sous espaces hilbertiens d'espaces vectoriels
        topologiques et noyaux associes, J. Analyse Math., 13 (1964),
        115-256.

\bibitem{Sh76}
        Shmul'yan Yu.L. Theory of linear relations and space with an indefinite metric (Russ.)
        Funktsion. analiz i ego pril., 10 (1976), no. 1, 67-72.

\bibitem{Sh70}
        Shmul'yan Yu.L.Direct and inverse problems for resolvent
        matrices, Dopov. Akad. Nauk. Ukr. SSR, Ser. A (1970), no. 6,
        514-517.

\bibitem{Yu83}
        Yuditskij P.M. Lifting problem, Deposited in Ukr. NIINTI
        18.04.1983, no. 311-Uk-D83.


 }
\end{thebibliography}
\end{document}